\documentclass[reqno,a4paper]{amsart}
\usepackage{amsmath}
\usepackage{amssymb}
\usepackage{amsthm}
\usepackage{enumerate}
\usepackage{epsfig}
\usepackage{graphicx}

\newtheorem{theorem}{Theorem}
\newtheorem{proposition}{Proposition}
\newtheorem{problem}{Problem}
\newtheorem{definition}{Definition}

\begin{document}

\title{An Etude on One Sharygin's Problem}\thanks{The research is partially supported by NI13 FMI-002.}

\author{Boyan Zlatanov}

\maketitle

{\sc Abstract:} {\small By the methods of the synthetic geometry we investigate properties of objects generated from
a complete quadrangle and a line, which lies in its plane. We start with a problem from the book of Sharygin ``Problems in Plane Geometry''.
We generalize this problem with the help of Pappus, Desargues and Pascal's Theorems and we discover new concurrent lines, collinear points,
and conic sections.}

\section{Introduction}
\label{intro}
The book of Sharygin ``Problems in Plane Geometry'' \cite{RefSh} is a collection of interesting and various in difficulty problems. A challenge that can unfold an entire world can be even a small mathematical problem. The Dynamic Geometry Software (DGS) facilitates substantially the efforts of the mathematicians in this direction \cite{RefKTZ}, \cite{RefZ}. We would like to illustrate an evolution of the idea implemented in a small school problem (\cite{RefSh}, Problem 34, p. 72) that is based on projective and combinatorial methods with the help of DGS.
\section{Preliminary}
\label{sec:1}
The investigations in the present work relate the Euclidian model of the Projective plane, i.e.
the Euclidian plane complimented with its infinite points and its infinite line $\omega$.

\begin{theorem}\label{th1}(Pappus) Let be given two lines $g$ and $g^\prime$  . If  $A, B, C\in g$ and $A^\prime , B^\prime , C^\prime\in g$,
then the points $P=BC^\prime\cap CB^\prime$, $Q=AC^\prime\cap CA^\prime$, $R=AB^\prime\cap BA^\prime$ are collinear.
\end{theorem}
A triangle is called the set of three noncollinear points and their three joining lines.
\begin{theorem}\label{th2}(Desargues)
The connecting lines of the couples of corresponding vertices of two triangles $ABC$ and $A^\prime B^\prime C^\prime$ are intersecting at a point $S$ if and only if the intersection points of the couples of corresponding sides $P=BC\cap BC$, $Q=AC\cap AC$, $R=AB\cap AB$ lie on a line $s$.
\end{theorem}
Two triangles that satisfied the conditions of Theorem \ref{th2} are called perspective. The point $S$ is called perspective center and the line $s$ is called a perspective axis.
\begin{theorem}\label{th3}(Pascal)
A hexagon $AB^\prime CA^\prime BC^\prime$ is inscribed in a conic section $k$ if and only if the
points $P=AB^\prime\cap A^\prime B$, $Q=B^\prime C\cap BC^\prime$, $R=CA^\prime\cap C^\prime A$ are collinear.
\end{theorem}

The line that is incident with the points $P, Q, R$ is called Pascal's line.

A complete quadrangle is called the set of four points $P$, $Q$, $R$, $S$, of which no three are collinear, and the lines $QR$, $PS$, $RP$, $QS$, $PQ$, $RS$.
The intersections of the opposite sides $A=QR\cap PS$, $B=RP\cap QS$, $C=PQ\cap RS$ are called diagonal points and they are the vertices of the diagonal triangle of the complete quadrangle $PQRS$.

\begin{theorem}\label{th4}(Pappus--Desargues)
The three pairs of opposite sides of a complete quadrangle intersect a line, which is not incident with any of its vertices, in three pairs of corresponding points for one and the same involution.
\end{theorem}

\begin{theorem}\label{th5}(\cite{RefC2}, Theorem 6.43, p. 73)
The diagonal triangle of every inscribed in a conic section complete quadrangle is self--polar.
\end{theorem}

The present investigation is inspired by one problem of Sharygin. Our results can be considered as improvisations
on the idea, that is encoded into the next problem.

\begin{problem}\label{pr1} (\cite{RefSh}, problem 34, p. 72)
Let $ABCD$ be a quadrangle and the points $M\in AC$ and $B\in BD$ be such that $BM\parallel AD$
and $AN\parallel BC$. Prove that $MN\parallel CD$.
\end{problem}
\begin{figure}
\epsfig{file=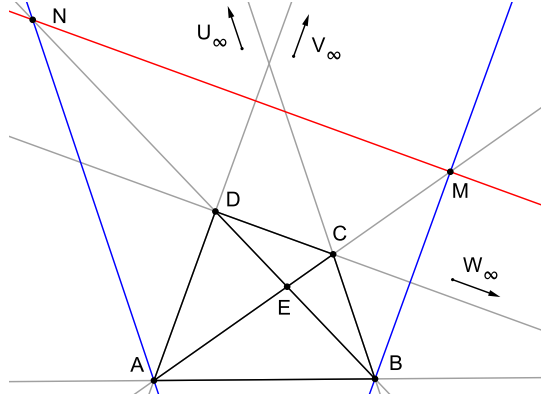,width=0.45\textwidth}
\caption{Sharygin's Problem}
\label{fig:4}
\end{figure}
The proof given by Sharygin is based on the proportionality of the corresponding sides of the similar triangles
$EMB$ and $EAD$, $EBC$ and $ENA$, where $E=AC\cap BD$ and Thales' Theorem.

\section{A Generalized Sharygin's Problem}
\label{sec:2}

We would like to present another solution of Problem \ref{pr1}, that is based on Pappus' Theorem. This solution will
help us to investigate some properties of a quadrangle and a line which lies in its plane.
We would like to mention that a similar approach have been used for generalizing of another Sharygin's Problem in (\cite{RefKTZ}, Problems 12 and 13).
\begin{proof}
Let consider the quadrangle $ABCD$ in the extended Euclidian plane and let us introduce the notations: $U_\infty =BC\cap\omega$, $V_\infty =AD\cap\omega$, $W_\infty =CD\cap\omega$ (Fig. \ref{fig:4}).
Let us consider the ordered triads of collinear points $(B, C, U_\infty )$ and $(A, V_\infty , D)$.
According to Theorem \ref{th1} it follows that the points $W_\infty =CD\cap U_\infty V_\infty =CD\cap\omega$, $M=BV_\infty\cap AC$ and $N=AU_\infty\cap BD$ are collinear, which means that $W_\infty\in MN$, i.e. $MN\parallel CD$.
\end{proof}

The proof, which we have presented is not only shorter, but it gives a possibility to develop the idea encoded into this small school Sharygin's problem. We can generate $48$ similar problems band together with a common simple solution. We will state Problem \ref{pr1} in the language of the extended Euclidian plane.

\noindent{\bf Problem \ref{pr1}$^\ast$}
Let $ABCD$ be a quadrangle and $U_\infty$ be the infinite point on the line $BC$ and $V_\infty$ be the infinite point on the line $AD$. Prove that the points $M=BV_\infty\cap AC$, $N=AU_\infty\cap BD$ and $W_{\infty}=CD \cap \omega$ are collinear, which means that $MN\parallel CD$.

We notice in the proof of Problem \ref{pr1}, that the infinite points $U_\infty$ and $V_\infty$ can be replaced with finite points $U\in BC$ or $V\in AD$. Thus we can state a variant of Problem \ref{pr1}$^\ast$.
\begin{problem}\label{pr2}
Let $ABCD$ be a quadrangle
and $U\in BC$ and $V\in AD$ be arbitrary points. Prove that the points $M=BV\cap AC$, $N=AU\cap BD$ and $W=CD\cap UV$ are collinear, i.e. the lines $MN$, $CD$, $UV$ are concurrent.
\end{problem}
\begin{figure}
\epsfig{file=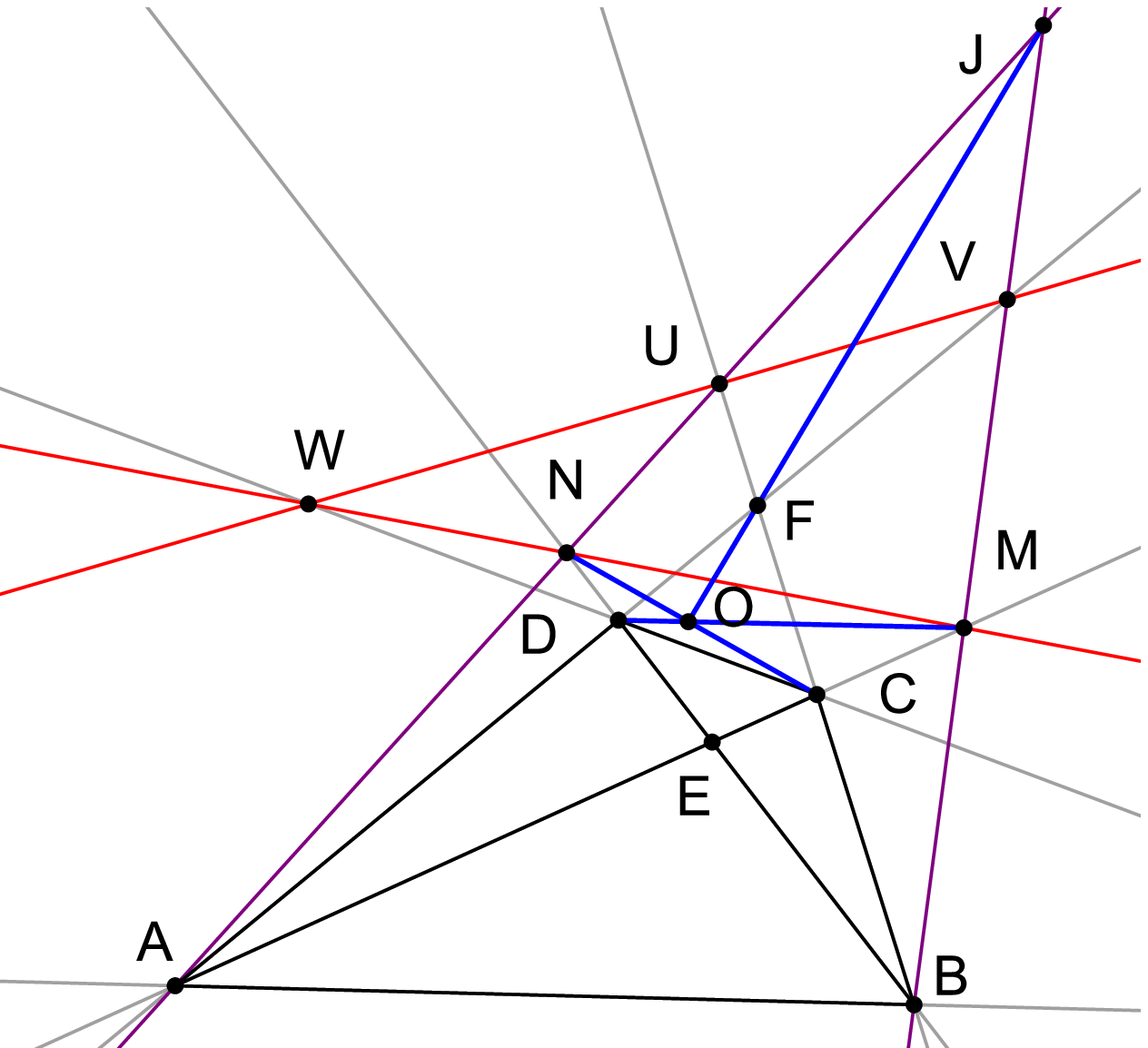,width=0.35\textwidth}\hfill\epsfig{file=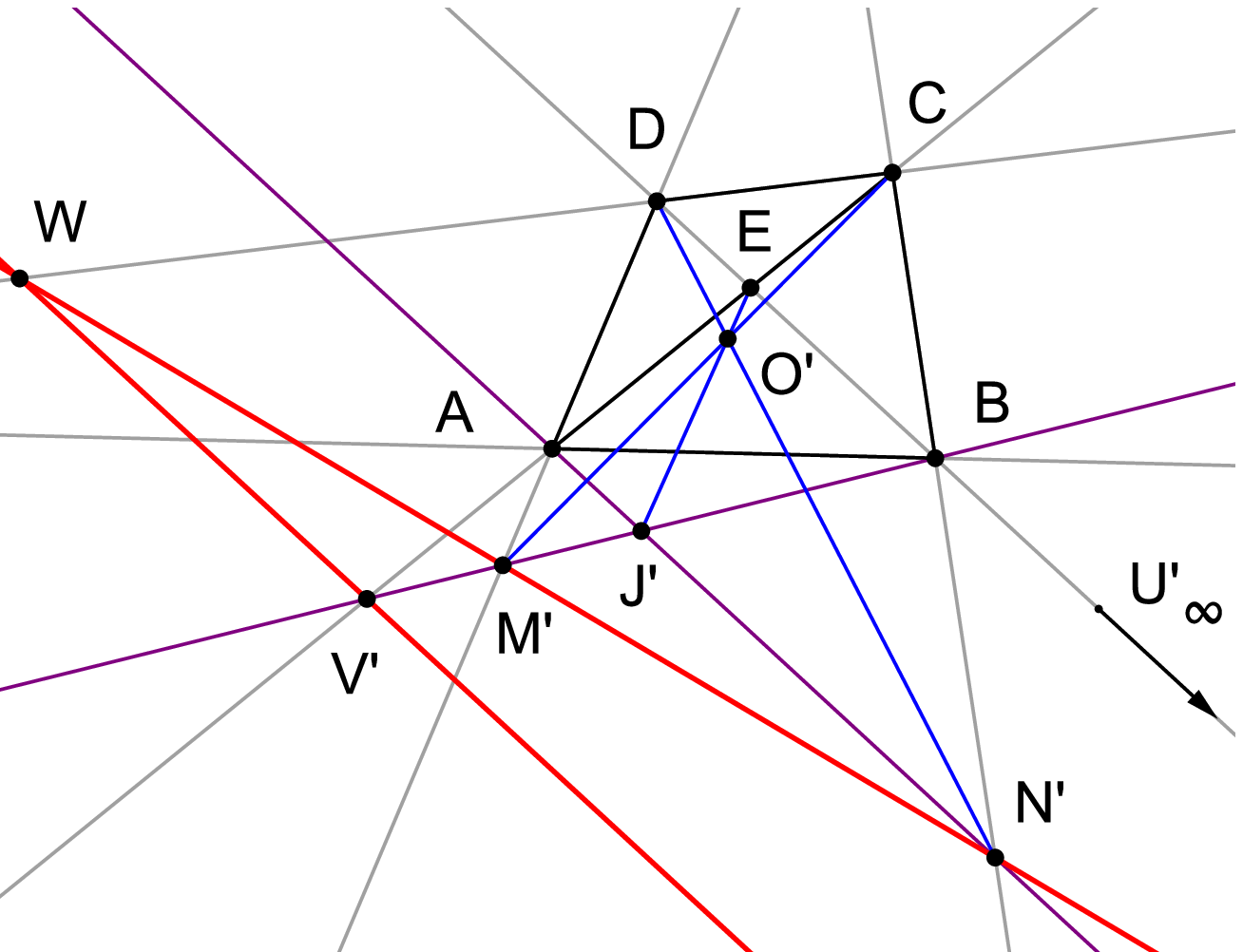,width=0.40\textwidth}
\caption{{\it Left} Sharygin's Problem for finite points $U$ and $V$. {\it Right} Sharygin's Problem for one finite point $V$ and one infinite point $U_\infty$}
\label{fig:6}
\end{figure}
\begin{proof} It is enough to apply Theorem \ref{th1} for the ordered triads of points $(B, C, U)$ and $(A, V , D)$ (Fig. \ref{fig:6}, Left) . 	
\end{proof}
Since there exists two possibilities for the points $U$ and $V$ -- to be finite or infinite, then their possible combinations are four.
We observe that the role of pair of lines $(BC, AD)$ and $(AC, BD)$ can be interchanged in
Problem \ref{pr1} and Problem \ref{pr1}$^\ast$. Let us state another variant of Sharygin's Problem.

\begin{problem}\label{pr3}
Let $ABCD$ be a quadrangle
and $V^\prime\in AC$ be an arbitrary point and $U^\prime_\infty$ is the infinite point on the line $BD$. Let us denote $M^\prime=BV^\prime\cap AD$, $N^\prime=AU^\prime_\infty\cap BC$ . Prove that the lines $CD$, $V^\prime U^\prime_\infty$ and $M^\prime N^\prime$ are concurrent
(Fig. \ref{fig:6}, Right).
\end{problem}
\begin{proof} It is enough to apply Theorem \ref{th1} for the ordered triads of points $(A, C, V^\prime)$ and $(B, U^\prime_\infty , D)$.
\end{proof}

The Dynamic Geometry Softwares facilitate the research work. They can help us to state a hypotheses, which are necessary to be proven after that with synthetic methods or with ACS \cite{RefZ}. Thus with the notations in Fig. \ref{fig:6} to the left we can state a hypothesis: {\it The lines $DM$, $CN$, $FJ$ are concurrent.} Using the notation in Fig. \ref{fig:6} to the right we can state the hypothesis: {\it The lines $DN^\prime$, $CM^\prime$, $EJ^\prime$ are concurrent.} We will prove the first statement. Let us consider the triangles $DCF$ and $MNJ$. From Problem \ref{pr2}
the points $W=DC\cap MN$, $U=CF\cap NJ$ and $V=DF\cap MJ$ are collinear.
According to Theorem \ref{th2} the triangles $DCF$ and $MNJ$  are perspective ones with a perspective axis $UV$. Therefore the connecting lines of the pairs of corresponding vertices i.e. $DM$, $CN$,  $FJ$ are concurrent.

We will replace the quadrangle $ABCD$ with the complete quadrangle $A_1A_2A_3A_4$ in order to be able to state all of the cases in one problem.
\begin{problem}\label{pr4}(Generalized Sharygin's problem)
Let $A_1A_2A_3A_4$ be a complete quadrangle, $A_i$, $A_j$ be arbitrary pair of its vertices and the points $U_{js}\in A_jA_s$ and $U_{ik}\in A_iA_k$ be arbitrary chosen, where $i, j, k, s\in \{1,2,3,4\}$ and any two of them are different. Let us denote $g=U_{js}U_{ik}$, $I=A_iA_k\cap A_jA_s$, $M=A_iU_{js}\cap A_jA_k$, $N=A_jU_{ik}\cap A_iA_s$ and $J=A_iM\cap A_jN$.
Prove that:\\
1) The lines $g$, $MN$, $A_kA_s$ are concurrent;\\
2) The lines $A_sM$, $A_kN$, $IJ$ are concurrent.
\end{problem}
Let us point out that the points $U_{is}$ and $U_{si}$ coincide for any $i,s\in\{1,2,3,4\}$, $i\not= s$.
Therefore we will use the notation $U_{is}$, $i<s$ for all the figures and examples.
\begin{proof}
1) We apply Theorem \ref{th1} to the ordered triads of collinear points $(A_i,A_k,$ $U_{ik})$ and $(A_j, U_{js}, A_s)$ and we get that the points $M$, $N$ and $W=g\cap A_kA_s$ are collinear, i.e. the lines $g$, $MN$, $A_kA_s$ are concurrent.

2) Let us consider the triangles $A_sA_kI$ and $MNJ$. We establish that $A_sI\cap MJ=A_sA_j\cap A_iU_{js}=U_{js}\in g$ and $A_kI\cap NJ=A_kA_i\cap NA_j=A_kA_i\cap A_jU_{ik}=U_{ik}\in g$. From 1) we have $A_kA_s\cap NM=W\in g$. Consequently the triangles $A_sA_kI$ and $MNJ$ are perspective with a perspective axis $g$. According to Theorem \ref{th2} the lines $A_sM$, $A_kN$, $IJ$ are concurrent.
\end{proof}

\begin{definition}\label{d1}
If $A_1A_2A_3A_4$ is a complete quadrangle and the points $U_{js}\in A_jA_s$ and $U_{ik}\in A_iA_k$ are arbitrary chosen, then the points $M=A_iU_{js}\cap A_jA_k$ and $N=A_jU_{ik}\cap A_iA_s$, will be called Sharygin's points for $A_1A_2A_3A_4$, which are associate
with the line $g=U_{js}U_{ik}$ and the vertices $A_i$, $A_j$.
\end{definition}

There are $C^2_4=6$ possible choices for the pairs of vertices $(A_i, A_j)$. For any chosen pair $(A_i, A_j)$ there are two different problems, because of the existence of two mutually exclusive possibilities for the indices $s$ and $k$. The combinations of the points $U_{js}$ and $U_{ik}$
are four, because any one can be finite or infinite. Therefore the number of the specific tasks that can be formulated from Problem \ref{pr4} is $6.2.4=48$.\\
For example Problem \ref{pr1} is a particular case of the problem \ref{pr4} for $i=1$, $j=2$, $s=3$, $k=4$,
where $U_{23}$ and $U_{14}$ are the infinite points of the lines $A_2A_3$ and $A_1A_4$, respectively.

The visualization of the particular cases of Problem \ref{pr4} can be done easily with the help of the special function ``Swap finite and infinite points'' in DGS -- Sam \cite{RefKTZ}. It is enough to sketch the problem for one of four combinations of the points $U_{js}$ and $U_{ik}$ and the other three
are obtained with the help of the function ``Swap finite and infinite points''.

For the convenience of the reader we will state a particular case of Problem \ref{pr4} for $i=1$, $j=3$ in the next problem.
There are two choices for the indices $s$ and $k$: $s=2$, $k=4$ or $s=4$, $k=2$. That is why one can see four different Sharygin's points,
associated with the vertices $A_1$ and $A_3$.

\noindent\textbf{Problem \ref{pr4}$^\ast$}
Let $A_1A_2A_3A_4$ be a complete quadrangle, $A_1$, $A_3$ be a pair of its vertices. Let the finite points $U_{23}\in A_2A_3$, $U_{14}\in A_1A_4$, $U_{34}\in A_3A_4$, $U_{12}\in A_1A_2$ be arbitrary chosen. Let us denote $I=A_1A_4\cap A_3A_2$,  $M=A_1U_{23}\cap A_3A_4$, $N=A_3U_{14}\cap A_1A_2$, $J=A_1M\cap A_3N$, $M^\prime =A_1U_{34}\cap A_3A_2$, $N^\prime =A_3U_{12}\cap A_1A_4$, $J^\prime=A_1 M^\prime\cap A_3N^\prime$ (Fig. \ref{fig:9}) .
Prove that:\\
1) The lines $U_{23}U_{14}$, $MN$, $A_2A_4$ are concurrent; the lines $U_{34}U_{12}$, $M^\prime N^\prime$, $A_2A_4$ are concurrent;\\
2) The lines $A_2M$, $A_4N$, $IJ$  are concurrent; the lines  $A_2N^\prime$, $A_4M^\prime$, $I^\prime J^\prime$ are concurrent.

\begin{figure}
\epsfig{file=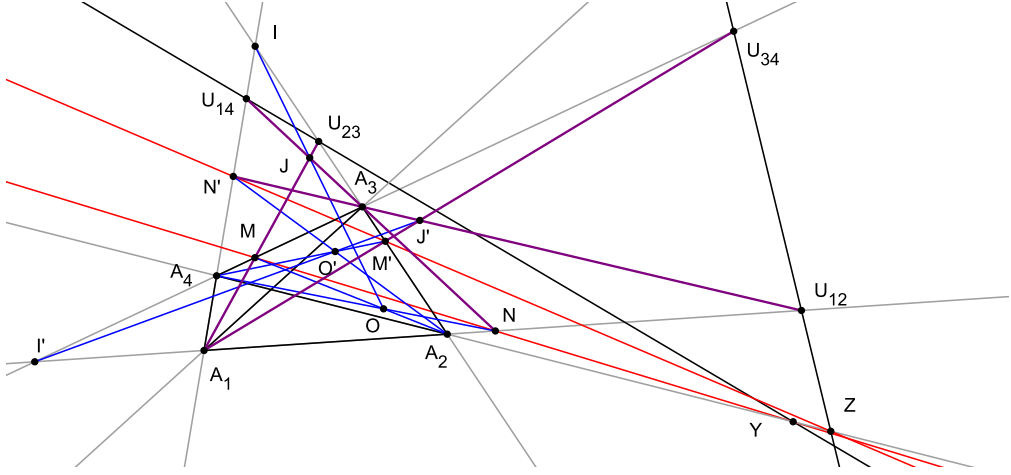,width=0.85\textwidth}
\caption{The four Sharygin's points for $i=1$, $j=3$}
\label{fig:9}
\end{figure}

\section{Sharygin's Conic Sections}
\label{sec:3}

We will put an additional condition on the points $U_{ij}\in A_iA_j$, $i,j\in\{1,2,3,4\}$, $i\not=j$ to be collinear for the next investigations.

\begin{definition}\label{d2}
Let $A_1A_2A_3A_4$ be a complete quadrangle and let $g$ be a line, which lies in the plane of $A_1A_2A_3A_4$ and it is not incident with any
 vertex of $A_1A_2A_3A_4$ or with any of its diagonal points. We call the pair $(A_1A_2A_3A_4,g)$ a $(q,l)$--pair.
\end{definition}

For any $(q,l)$--pair $(A_1A_2A_3A_4,g)$ we will use in the sequel the notation
\begin{equation}\label{e0}
U_{ij}=U_{ji}=g\cap A_iA_j, i,j\in\{1,2,3,4\}, i\not=j.
\end{equation}

According to Theorem \ref{th4} the pairs of points $(U_{12},U_{34})$, $(U_{13},U_{24})$ and $(U_{14},U_{23})$ are pairs of conjugated
points for one and the same involution $\varphi$. The additional condition imposed on the line $g$ not to be incident with a diagonal point
excludes the possibility the Sharygin's points to coincide with vertices of $A_1A_2A_3A_4$.

When the points $U_{ij}$ are collinear it will be easier to introduce a unified notation for the Sharygin's points.
The two Sharygin's points associated with the vertices $A_i$, $A_j$ and the points $U_{js}$, $U_{ik}$ and the
two Sharygin's points associated with the vertices $A_i$, $A_j$ and the points $U_{jk}$, $U_{is}$ are determined
by the pair of vertices $A_i$, $A_j$ and the line $g$. Thus we can denote the four Sharygin's points associated with the pair of vertices
$A_i$, $A_j$ of the $(q,l)$--pair $(A_1A_2A_3A_4,g)$ with
\begin{equation}\label{e1}
\begin{array} {l}
M_{ij}^k =A_iU_{js}\cap A_jA_k, \ M_{ji}^s=A_jU_{ik}\cap A_iA_s, \\[10pt]
M_{ij}^s =A_iU_{jk}\cap A_jA_s, \ M_{ji}^k =A_jU_{is}\cap A_iA_k.
\end{array}
\end{equation}
Four Sharygin's points are connected to any pair of vertices for any $(q,l)$--pair $(A_1A_2A_3A_4,g)$ and they are the intersecting
points of the connecting lines of the vertices $A_i$ and $A_j$ with the pair of conjugated points for the Pappus--Desargues involution,
which is induced by $A_1A_2A_3A_4$ into $g$, with the third pair of opposite sides.
Let us introduce the points
\begin{equation}\label{e2}
\begin{array} {lr}
I=A_iA_k\cap A_jA_s, \ \ \ I^\prime=A_iA_s\cap A_jA_k; \\[10pt]
J=A_i U_{js}\cap A_j U_{ik}, \ J^\prime=A_iU_{jk}\cap A_jU_{is}; \\[10pt]
L=A_k U_{js}\cap A_s U_{ik}, \ L^\prime=A_sU_{jk}\cap A_kU_{is},
\end{array}
\end{equation}
which we will need in the sequel.

\begin{theorem}\label{th6}
Let $(A_1A_2A_3A_4,g)$ be a $(q,l)$--pair. Then the following hold true:\\
1) The lines $M_{ij}^kM_{ji}^s$, $M_{ij}^sM_{ji}^k$, are incident with the point $U_{ks}$;\\
 The lines $M_{sk}^jM_{ks}^i$, $M_{sk}^iM_{ks}^j$, are incident with the point $U_{ij}$.  \\
2) The lines $A_iM_{ks}^j, A_jM_{sk}^i, M_{ij}^kA_s, M_{ji}^sA_k, JI, I^\prime L^\prime$ are concurrent;  \\
   The lines $A_iM_{sk}^j, A_jM_{ks}^i,  M_{ji}^k A_s, M_{ij}^s A_k, IL, J^\prime I^\prime$ are concurrent. \\
3) The lines $M_{ij}^kM_{ij}^s$, $M_{ji}^sM_{ji}^k$, $A_iA_j$, $IJ$, $I^\prime J^\prime$ are concurrent; \\
The lines $M_{sk}^jM_{sk}^i$, $M_{ks}^iM_{ks}^j$, $A_sA_k$, $IL$, $I^\prime L^\prime$,  are concurrent.
\end{theorem}

\begin{figure}
\epsfig{file=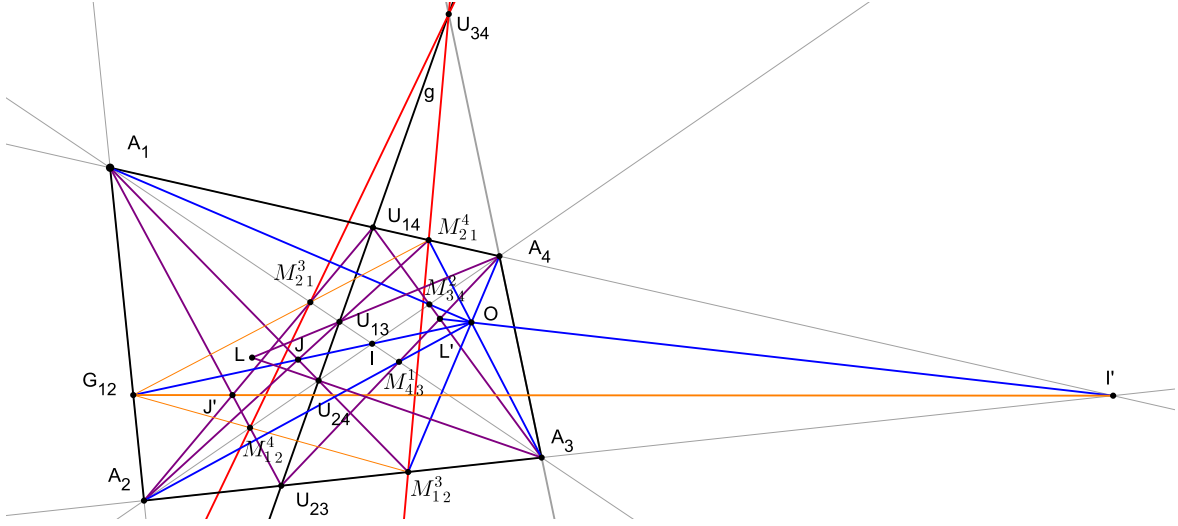,width=0.99\textwidth}
\caption{Theorem \ref{th6} for $i=1$, $j=2$, $k=3$, $s=4$}
\label{fig:11}
\end{figure}

\begin{proof}Just for a convenience of the reader we will write and the Sharygin's points associated with the pair of vertices
$A_s$, $A_k$:
\begin{equation}\label{e22}
\begin{array} {l}
M_{sk}^j=A_sU_{ik}\cap A_jA_k, \ \ M_{ks}^i=A_kU_{js}\cap A_iA_s, \\[10pt]
M_{sk}^i=A_sU_{jk}\cap A_iA_k, \ \ M_{ks}^j=A_kU_{is}\cap A_jA_s.
\end{array}
\end{equation}

1) We apply Theorem \ref{th1} for the ordered triads of points $(A_i,A_k,U_{ik})$ and $(A_j,U_{js},A_s)$ and using (\ref{e2}) we get that the
points $M_{ij}^k$, $M_{ji}^s$ and $A_kA_s\cap U_{ik}U_{js}$ are collinear. Therefore the lines $M_{ij}^kM_{ji}^s$, $A_kA_s$ and $g=U_{ik}U_{js}$ are incident with the point
$U_{ks}$. Applying Theorem \ref{th1} once more for the ordered triads of points $(A_i,A_s,U_{is})$ and $(A_j,U_{jk},A_k)$ and using (\ref{e2}) we get that the
points $M_{ij}^s$, $M_{ji}^k$ and $A_kA_s\cap U_{is}U_{jk}$ are collinear. Therefore the lines $M_{ij}^s M_{ji}^k$, $A_kA_s$ and $g$ are incident with the point $U_{ks}$. With the help of (\ref{e22}) and Theorem \ref{th1}, applied consequently for the ordered triads of points $(A_s,A_j,U_{sj})$ and $(A_k,U_{ik},A_i)$,  $(A_s,A_i,U_{si})$ and $(A_k,U_{jk},A_j)$, we prove that the lines $M_{sk}^jM_{ks}^i$ and $M_{sk}^i M_{ks}^j$ are incident with the point $U_{ij}$.

2) Let us consider the triangles $A_sA_kI$ and $M_{ij}^kM_{ji}^sJ$. Taking into account (\ref{e0}), (\ref{e1}), (\ref{e2}) and 1) in Theorem \ref{th6} we get that the points $A_sA_k\cap M_{ij}^kM_{ji}^s=U_{ks}$, $A_kI\cap M_{ji}^sJ=A_iA_k\cap A_jU_{ik}=U_{ik}$ and $A_sI\cap M_{ij}^kJ=A_jA_s\cap A_iU_{js}=U_{js}$ are collinear. Then from Theorem \ref{th2} it follows that the lines $A_sM_{ij}^k$, $A_kM_{ji}^s$ and $IJ$ are concurrent and let's denote $O=A_sM_{ij}^k \cap A_kM_{ji}^s \cap IJ$. Taking into account (\ref{e0}), (\ref{e1}), (\ref{e2}) and 1) in Theorem \ref{th6} we get again that the triangles $A_s A_k L^\prime$ and $M_{ij}^kM_{ji}^sI^\prime $ are perspective with a perspective axis $g$. Therefore these triangles have a perspective center $O^a=A_sM_{ij}^k \cap A_kM_{ji}^s \cap L^\prime I^\prime$. It is easy to see that $O=O^a$. Using (\ref{e0}), (\ref{e1}), (\ref{e2}), (\ref{e22}) and 1) in Theorem \ref{th6} we establish that the pairs of triangles $A_i A_j J$ and $M_{ks}^j M_{sk}^i I$, $A_i A_j I^\prime$ and $M_{ks}^j M_{sk}^i L^\prime$ are perspective with a perspective axis $g$. Hence $A_i M_{ks}^j \cap A_j M_{sk}^i \cap JI=O^b$ and $A_i M_{ks}^j \cap A_j M_{sk}^i \cap I^\prime L^\prime =O^c$. It is easy to observe that $O^b=O^c$ and then $A_i M_{ks}^j \cap A_j M_{sk}^i \cap I^\prime L^\prime \cap JI =O^c$.
Comparing all results we get $O=O^a=O^b=O^c$ or
\begin{equation}\label{e3}
A_iM_{ks}^j \cap A_jM_{sk}^i \cap A_sM_{ij}^k \cap A_kM_{ji}^s \cap JI \cap I^\prime L^\prime = O .
\end{equation}
 By similar considerations for the pairs of triangles $A_kA_sI^\prime$ and $M_{ij}^s M_{ji}^k J^\prime$, $A_kA_s L$ and  $M_{ij}^s M_{ji}^k I$ we prove that they have a common perspective center $O^\prime=A_k M_{ij}^s \cap  A_s M_{ji}^k \cap I^\prime J^\prime \cap LI$. Repeating the considerations for the pairs of triangles $A_iA_jI$ and $M_{sk}^jM_{ks}^iL$, $A_iA_jJ^\prime$ and $M_{sk}^jM_{ks}^iI^\prime$ we get that they have a common perspective center $\hat{O}=A_i M_{sk}^j \cap  A_j M_{ks}^i \cap IL \cap J^\prime I^\prime $. Comparing these results we can write
\begin{equation}\label{e4}
A_i M_{sk}^j \cap  A_j M_{ks}^i  \cap A_s M_{ji}^k  \cap A_k M_{ij}^s  \cap I^\prime J^\prime \cap LI =O^\prime .
\end{equation}
3) Let us consider the triangles $M_{ij}^kA_jJ$ and $M_{ij}^s A_iI$. Using (\ref{e0}), (\ref{e1}) and (\ref{e2}) we get that the intersecting points $U_{jk}$, $U_{ik}$, $U_{js}$ of their corresponding sides lie on the line $g$. Therefore the triangles $M_{ij}^kA_jJ$ and $M_{ij}^s A_iI$ are perspective with a perspective axis $g$. According to Theorem \ref{th2} the lines $M_{ij}^kM_{ij}^s$, $A_jA_i$, $JI$ are passing through a common point. Let us denote
$ M_{ij}^kM_{ij}^s\cap A_jA_i\cap JI=G_{ij}.$

After similar considerations for the triangles $M_{ji}^sA_iJ$ and $M_{ji}^k A_jI$ we prove that they are perspective with a perspective axis $g$. According to Theorem \ref{th2} they have a  perspective center and we can write $M_{ji}^sM_{ji}^k\cap A_iA_j\cap JI=\widehat{G}_{ij}.$
Therefore $G_{ij}=\widehat{G}_{ij}$.

At the end let us consider the triangles $M_{ij}^s A_jJ^\prime$ and $M_{ij}^kA_iI^\prime$. They are perspective again with a perspective axis $g$. Therefore the connecting lines $M_{ij}^kM_{ij}^s$, $A_iA_j$, $J^\prime I^\prime$ of their corresponding vertices are concurrent lines and it is easy to see they are passing through the point $G_{ij}$. Thus we received
\begin{equation}\label{e5}
M_{ij}^kM_{ij}^s\cap M_{ji}^sM_{ji}^k  \cap A_iA_j\cap JI \cap J^\prime I^\prime=G_{ij}.
\end{equation}
Now with the help of the pairs of triangles $M_{sk}^j A_s I^\prime$ and $M_{sk}^i A_k L^\prime$, $M_{ks}^i A_sI^\prime$ and $M_{ks}^j A_k L^\prime$,
$M_{sk}^i A_s I$ and $M_{sk}^j A_k L$, (\ref{e0}), (\ref{e1}), (\ref{e2}) and (\ref{e22}) we obtain
\begin{equation}\label{e6}
M_{sk}^jM_{sk}^i\cap M_{ks}^iM_{ks}^j  \cap A_sA_k\cap LI \cap L^\prime I^\prime=G_{sk}.
\end{equation}
\end{proof}

We present a particular case of Theorem \ref{th6} for $i=1$, $j=2$, $k=3$, $s=4$ on Fig. \ref{fig:11}.

By Theorem \ref{th6} it follows that for any given $(q,l)$--pair $(A_1A_2A_3A_4, g)$ there exists one point $G_{ij}\in A_iA_j$ for every side $A_iA_j$
with the properties (\ref{e5}).

\begin{definition}\label{d3}
Let $(A_1A_2A_3A_4,g)$ be a $(q,l)$--pair. The intersection point of the line $A_iA_j$ with
the line $M_{ij}^sM_{ij}^k$ will be denoted with $G_{ij}$ and will be called a $G_{ij}$--point for the
pair $(A_1A_2A_3A_4,g)$, associated with the vertices $A_i$ and $A_j$.
\end{definition}

\begin{proposition}\label{prop1}
The $G_{ij}$--point is a harmonic conjugate point of the point $U_{ij}$ with respect to the pair $A_i, A_j$.
\end{proposition}
\begin{proof}
Let us consider the complete quadrangle $M_{ij}^sM_{ij}^k U_{js}U_{jk}$ which vertices are defined by (\ref{e0}) and (\ref{e1}). The points $A_j=M_{ij}^sU_{js}\cap M_{ij}^k U_{jk}$ and $A_i=M_{ij}^sU_{jk}\cap M_{ij}^k U_{js}$ are diagonal points for $M_{ij}^sM_{ij}^k U_{js}U_{jk}$. From (\ref{e0}) and (\ref{e5})
it follows that the pair of opposite sides $M_{ij}^sM_{ij}^k$ and $U_{js}U_{jk}$, which are passing through the
third diagonal point, intersect the side $A_iA_j$ in the points $G_{ij}$ and $U_{ij}$, respectively.
So, following \cite{RefC1} we can write the harmonic group $H(G_{ij}U_{ij}, A_i A_j).$
\end{proof}
Theorem \ref{th6} holds true also if the line $g$ is replaced with the infinite line $\omega$.
In this case the points $U_{ij}$ are infinite points wherefore the points $G _{ij}$ are midpoints for the Euclidean segments $A_iA_j$.
The new element in the proof of Theorem \ref{th6} will be a simplification in the proof of 3). Indeed, now the quadrangles
$A_iJA_jI$, $A_iJ^\prime A_jI^\prime$, $A_iM_{ij}^sA_jM_{ij}^k$, $A_iM_{ji}^sA_jM_{ji}^k$ are parallelograms and it follows that the segments $IJ$,
$I^\prime J^\prime$, $M_{ij}^sM_{ij}^k$, $M_{ji}^sM_{ji}^k$ and $A_iA_j$ have a common midpoint $G_{ij}$.
The case for $i=1$ and $j=3$ is presented in Fig. \ref{fig:122} to the right.
\begin{proposition}\label{prop2}
Any two vertices and the four  Sharygin's points that are connected with them lie on a conic section.
\end{proposition}
\begin{proof}
Let us consider the hexagon $A_iM_{ij}^kM_{ij}^s A_jM_{ji}^sM_{ji}^k$. From (\ref{e1}), (\ref{e2}) and 3) in Theorem \ref{th6}  it follows that the points
$A_iM_{ij}^k\cap A_jM_{ji}^s=A_iU_{js}\cap A_jU_{ik}=J$, $M_{ij}^kM_{ij}^s \cap M_{ji}^sM_{ji}^k=G_{ij}$ and $M_{ij}^s A_j\cap M_{ji}^kA_i=
A_jA_s\cap A_iA_k=I$
are collinear. Then according to Theorem \ref{th3} the hexagon $A_iM_{ij}^sM_{ij}^s A_jM_{ji}^sM_{ji}^k$ is inscribed in a conic section.
\end{proof}
\begin{figure}
 \epsfig{file=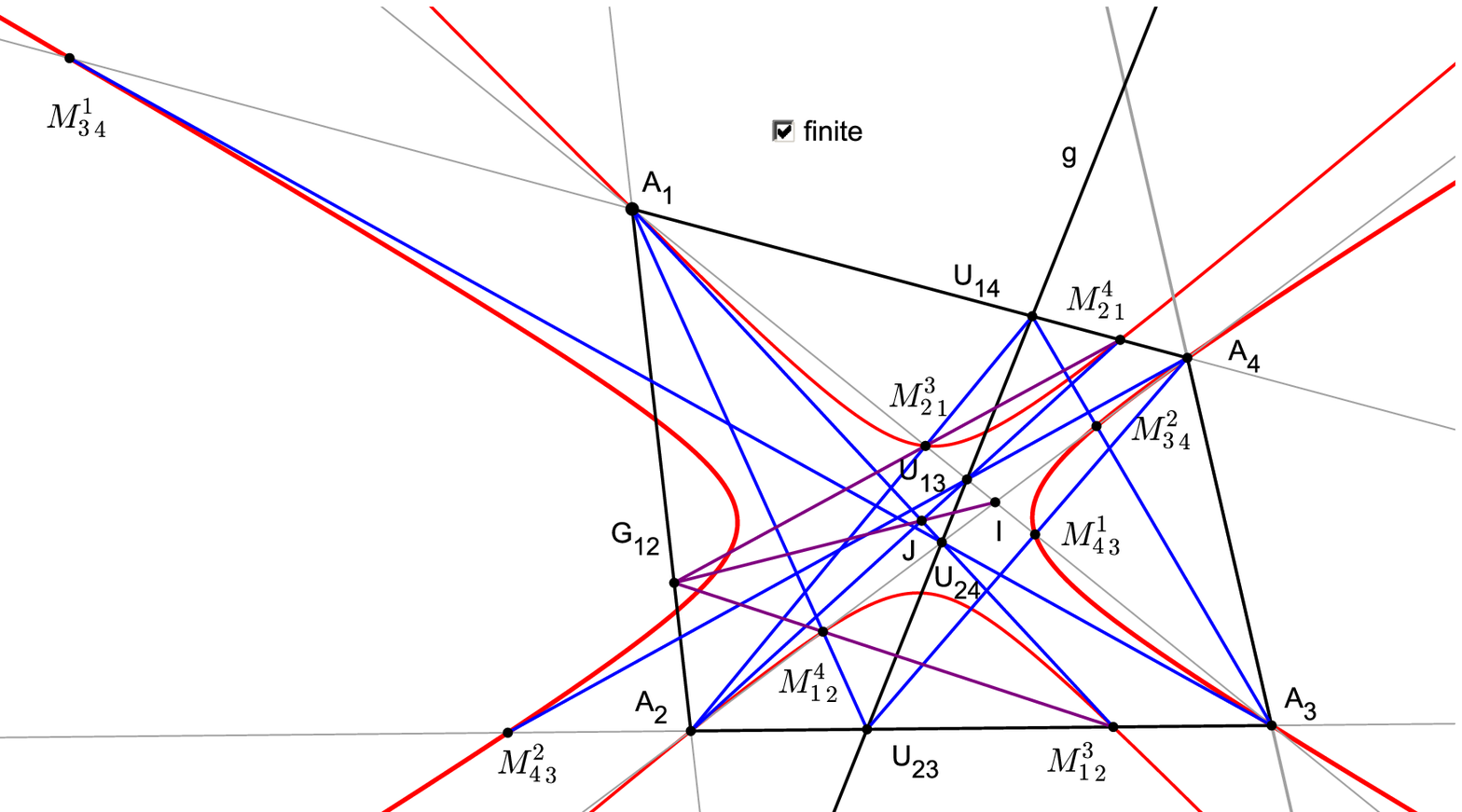,width=0.64\textwidth}\hfill\epsfig{file=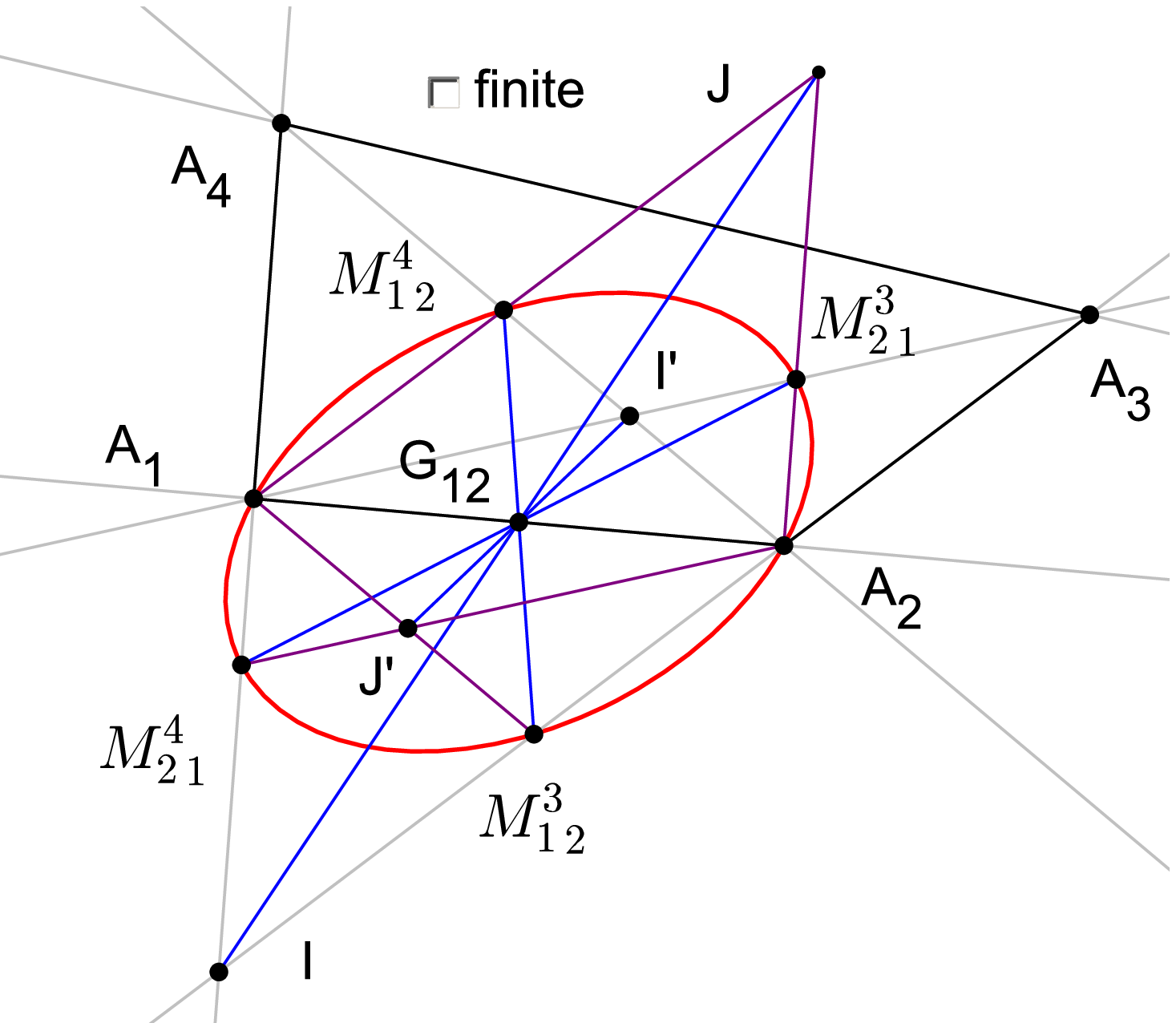,width=0.34\textwidth}
\caption{{\it Left} Sharygin's curves $k_{12}$ and $k_{34}$, when $g$ is a finite line. {\it Right} Sharygin's curve $k_{12}$, when $g$ is the infinite line $\omega$.}
\label{fig:122}
\end{figure}

\begin{definition}\label{d4}
The conic section, which passes through the vertices $A_i$ and $A_s$ of the complete quadrangle $A_1A_2A_3A_4$ from the $(q,l)$--pair
$(A_1A_2A_3A_4, g)$ and through the four Sharygin's points $M_{is}^j$, $M_{is}^k$, $M_{si}^j$, $M_{si}^k$ associated with these vertices will be called Sharygin's curve for the $(q,l)$--pair $(A_1A_2A_3A_4, g)$ and will be denoted with $k_{is}$.
\end{definition}

From Proposition \ref{prop2} it follows that there are six Sharygin's curves for any $(q,l)$--pair $(A_1A_2A_3A_4,g)$.
Fig. \ref{fig:122} to the left presents the Sharygin's curves $k_{12}$ and $k_{34}$, when $g$ is a finite line.
Fig. \ref{fig:122} to the right presents the Sharygin's curve $k_{12}$, when $g$ is the infinite line $\omega$.

We have presented a technique for GeoGebra to simulate the special function ``Swap finite and infinite points'' of DGS -- Sam in \cite{RefZ}.
With the help of this technique Fig. \ref{fig:122} to the right, after removing
the Sharygin's curve $k_{34}$ just for simplicity of notations, can be generated from Fig. \ref{fig:122} to the left and vice versa.

\begin{proposition}\label{prop3}
The  $G_{ij}$ --point is the pole of the line $g$ with respect to the  Sharygin's curve $k_{ij}$ for the $(q,l)$--pair $(A_1A_2A_3A_4, g)$.
\end{proposition}
\begin{proof}
Let us consider the inscribed into the conic section $k_{ij}$ complete quadrangle $A_i M_{ij}^k M_{ij}^s A_j$. Using (\ref{e0}), (\ref{e2})
and 3) in Theorem \ref{th6} we find that the diagonal triangle of $A_i M_{ij}^k M_{ij}^s A_j$ have the following vertices  $U_{js}=A_iM_{ij}^k\cap A_jM_{ij}^s =A_iU_{js}\cap A_jA_s$, $U_{jk}=AM_{ij}^s\cap A_jM_{ij}^k=A_iU_{jk}\cap A_jA_k$, $G_{ij}=A_iA_j\cap M_{ij}^kM_{ij}^s$.
According to Theorem \ref{th5} this triangle is self--polar. Therefore the pole of the line $g=U_{js}U_{jk}$ with respect to $k_{ij}$
is the point $G_{ij}$.
\end{proof}

\begin{proposition}\label{prop4}
Let $(A_1A_2A_3A_4, g)$ be a $(q,l)$--pair. If the points $A_i, A_j$ and $A_j,A_s$  are two pairs of the vertices of $A_1A_2A_3A_4$
and $G_{ij}\in A_iA_j$ and $G_{js}\in A_jA_s$ are the poles of the line $g$ with respect to Sharygin's curves $k_{ij}$ and $k_{js}$,
respectively, then the line $G_{ij}G_{js}$ passes through the point $U_{is}=g\cap A_iA_s$.
\end{proposition}
\begin{proof}
Let us first introduce the notations:
\begin{equation}\label{e7}
\begin{array}{lr}
\overline{I}=A_iA_j \cap A_sA_k \ , \ \ \ \check{I}=A_iA_s \cap A_jA_k =I^{\prime} \ , \\[10pt]
\overline{J}=A_jU_{sk} \cap A_sU_{ij} \ , \ \ \ \check{J}= A_j U_{is} \cap A_s U_{jk}\ .
\end{array}
\end{equation}
From Theorem \ref{th6} it follows that we can define the point $G_{js}$ (a pole of $g$ with respect to Sharigin's curve $k_{js}$) by the
following way:
\begin{equation}\label{e8}
G_{js} = \overline{I}\ \overline{J} \cap A_jA_s \cap \check{I} \check{J}=
\overline{I}\ \overline{J} \cap A_jA_s \cap I^\prime \check{J}.
\end{equation}
Let us consider the triangles $G_{ij}A_iJ^\prime$ and $G_{js}A_s \check{J}$.
Using (\ref{e1}), (\ref{e7}) and (\ref{e8}) we find the intersection points of the pairs of corresponding sides:
$G_{ij}A_i\cap G_{js}A_s=A_iA_j\cap A_jA_s=A_j$,
$G_{ij}J^\prime \cap G_{js} \check{J}=I^\prime J^\prime \cap I^\prime \check{J}=I^\prime,$
$A_iJ^\prime \cap A_s \check{J}= A_i U_{jk} \cap A_s U_{jk}= U_{jk}.$
Since they lie on the side $A_jA_k$ of the complete quadrangle $A_1A_2A_3A_4$, the triangles $G_{ij}A_iJ^\prime$ and $G_{js}A_s \check{J}$ are perspective and according to Theorem \ref{th2} the connecting lines of the pairs of corresponding vertices
$G_{ij}G_{js}$, $A_jA_s$ and $J^\prime \check{J}$ are incident with one point-- the perspective center. With the help of (\ref{e0}), (\ref{e1}) and
(\ref{e7})we get $A_iA_s\cap J^\prime \check{J} = A_iA_s\cap A_jU_{is}=U_{is}$, from where it follows that the perspective center is the point $U_{is}$. Therefore the line $G_{ij}G_{js}$ passes through the point $U_{is}$.
\end{proof}
\begin{figure}
\epsfig{file=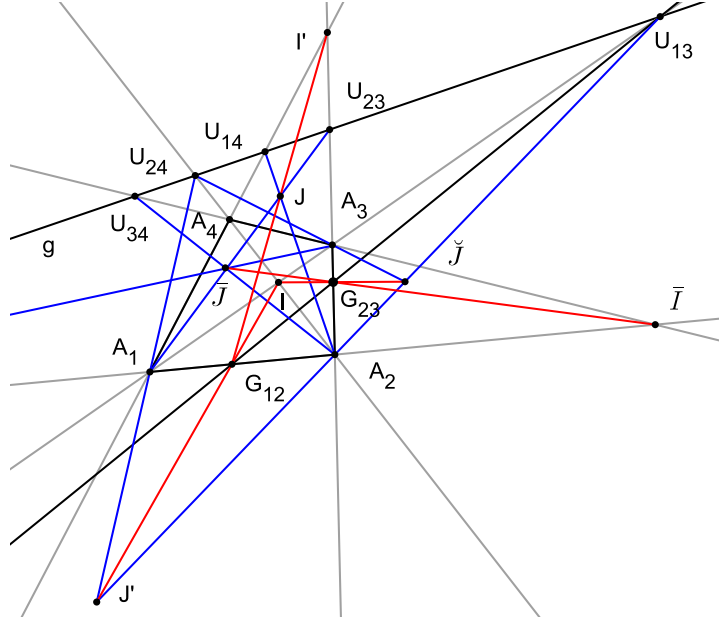,width=0.60\textwidth}
\caption{Proposition \ref{prop4} for $G_{12}$ and $G_{23}$}
\label{fig:15}
\end{figure}
A particular case of Proposition \ref{prop4} is presented on Fig. \ref{fig:15} for  $i=1$, $j=2$, $s=3$ and $k=4$.

Proposition \ref{prop4} presents us an easier technique for finding the remaining five $G_{ij}$--points,
once we have constructed one of the six points $G_{ij}$.

\begin{theorem}\label{th7}
Let $(A_1A_2A_3A_4, g)$ be a $(q,l)$--pair. Then the poles $G_{is}$ of the line $g$ with respect to the Sharygin's curves $k_{is}$, $i,s \in \{1,2,3,4\}$  $i\neq s$,  and the diagonal points of the quadrangle $A_1A_2A_3A_4$ lie on one conic section .
\end{theorem}
\begin{figure}
\epsfig{file=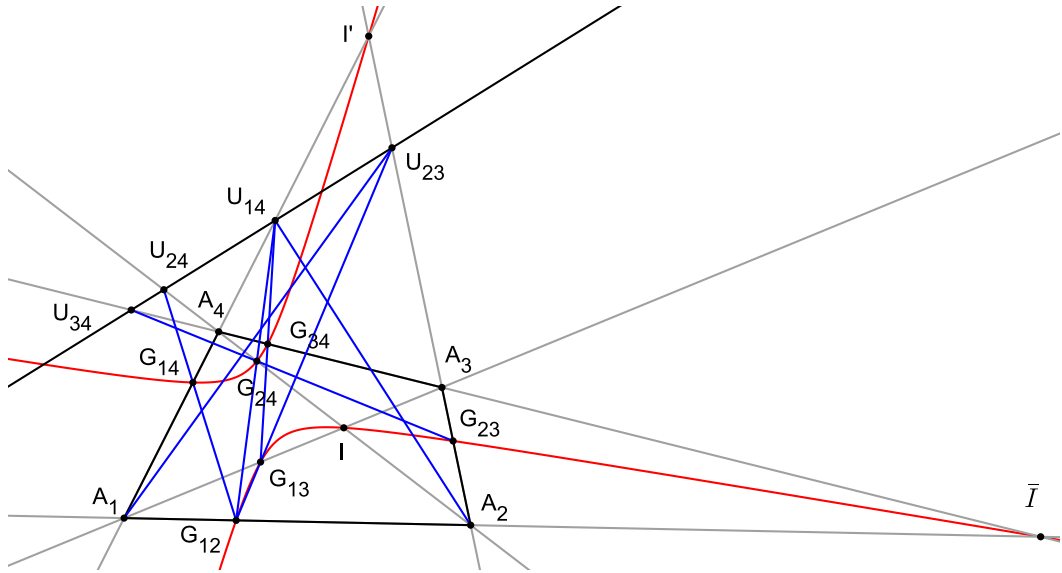,width=0.89\textwidth}
\caption{Nine points conic section}
\label{fig:14}
\end{figure}
\begin{proof}
Let consider the points $G_{13}$, $G_{23}$, $G_{12}$, $G_{24}$, $G_{14}$, $G_{34}$ (Fig. \ref{fig:14}). From Proposition \ref{prop4}
we get that $U_{12}=G_{13}G_{23}\cap G_{24}G_{14}$, $U_{13}=G_{23}G_{12}\cap G_{14}G_{34}$ and $U_{14}=G_{12}G_{24}\cap G_{34}G_{13}$.
By the condition $U_{is}\in g$ it follows that the points $U_{12}$, $U_{13}$ and $U_{14}$ are collinear. According to Theorem \ref{th3}
the hexagon $G_{13}G_{23}G_{12}G_{24}G_{14}G_{34}$ is inscribed into a curve of the second power $k$ and
$g$ is its Pascal's line.

Let consider the points $G_{13}$, $G_{23}$, $I$, $G_{14}$, $G_{24}$, $G_{34}$
Applying Proposition \ref{prop4} and using (\ref{e1})
we get the points $U_{12}=G_{13}G_{23}\cap G_{14}G_{24}$, $U_{23}=G_{23}I \cap G_{24}G_{34}=A_2A_3\cap G_{24}U_{23}$
and $U_{14}=I G_{14}\cap G_{34}G_{13}=A_1A_4\cap G_{34}U_{14}=U_{14}$.
By the condition (\ref{e0}) it follows that the points $U_{12}$, $U_{23}$ and $U_{14}$ are collinear. According to Theorem \ref{th3}
the hexagon $G_{13}G_{23}I G_{14}G_{24}G_{34}$ is inscribed into a curve of second power $k^\prime$ and
$g$ is its Pascal's line.

The curves $k$ and $k^\prime$ coincide because they have five common points.

By similar considerations for the points $G_{13}$, $I^\prime $, $G_{24}$, $G_{34}$, $G_{14}$, $G_{12}$
and for the points $G_{12}$, $\overline{I}$, $G_{34}$, $G_{14}$, $G_{24}$, $G_{23}$
it can be proved
that the hexagons $G_{13}I^\prime G_{24}G_{34}G_{14}G_{14}$ and $G_{12}\overline{I}G_{34}G_{14}G_{24}G_{23}$ are inscribed
into the curve of second power $k$.
\end{proof}

\begin{definition}
The curve $k$ on which lie the nine points $G_{12}$, $G_{13}$, $G_{14}$, $G_{23}$, $G_{24}$, $G_{34}$, $I$, $I^\prime$, $\overline{I}$
will be called the curve of the nine points for the $(q,l)$--pair $(A_1A_2A_3A_4, g)$.
\end{definition}

Some properties of the curve of the nine points, when the points $G_{ij}$ are midpoint of the Euclidian segments $A_iA_j$
have been investigated in \cite{RefZ}.

\begin{theorem}\label{th8}
The pole of $g$ with respect to the curve $k$ of the nine points for the $(q,l)$--pair $(A_1A_2A_3A_4, g)$ is the point
$G=G_{12}G_{34}\cap G_{13}G_{24}\cap G_{14}G_{23}$ .
\end{theorem}
\begin{proof}
Let consider the inscribed complete quadrangles $G_{12}G_{34}G_{13}G_{24}$ and $G_{12}G_{34}G_{14}G_{23}$ (Fig. \ref{fig:14}).
By Proposition \ref{prop4} their diagonal triangles are $U_{14} U_{23}G$, where $G=G_{12}G_{34}\cap G_{13}G_{24}$ and
$U_{13} U_{24}G^\prime$, where $G^\prime=G_{12}G_{34} \cap G_{14}G_{23}$, respectively. According to Theorem \ref{th5}
the pole of the line $U_{14}U_{23}=g$ with respect to the curve $k$ is the point $G$ and the pole of the line $U_{13}U_{24}=g$
with respect to the curve $k$ is the point $G^\prime$. Since the line $g$ is a polar of the points  $G $ and $G^\prime$ with respect to
the curve $k$, then $G=G^\prime$, which means that $G_{12}G_{34} \cap G_{13}G_{24}\cap G_{14}G_{23}=G=G^\prime$.
\end{proof}

If the line $g$ is the infinite line $\omega$ then $G$ is the center of the nine points curve $k$
and the $G_{ij}$--points are the centers of the Sharygin's curves $k_{ij}$ for the $(q,l)$--pair $(A_1A_2A_3A_4,g)$.

From Theorem \ref{th6} and Proposition \ref{prop4} it follows
\begin{proposition}\label{prop5}
There exists two homologies $ \Phi$ and $ \Phi^\prime$ with centers $O$ and $O^\prime$, respectively, for which the line $g$ is a common axis and which are related with every  pair of opposite sides for the quadrangle $A_1A_2A_3A_4$ from the
$(q,l)$--pair $(A_1A_2A_3A_4, g)$, such that:
\begin{equation}\label{e11}
\begin{array}{l}
\Phi \      (A_i, A_j, M_{ji}^s, M_{ij}^k; J, I^\prime) = M_{ks}^j, M_{sk}^i, A_k, A_s; I, L^\prime \ ;  \\ [10pt]
\Phi^\prime (A_i, A_j, M_{ij}^s, M_{ji}^k; J^\prime, I) = M_{sk}^j, M_{ks}^i, A_k, A_s; I^\prime, L \ .
\end{array}
\end{equation}
\end{proposition}
\begin{proof}
Using Theorem \ref{th6}, 
and the basic properties of any homology we can verify that the homologies
$\Phi(O, g; A_i \longrightarrow M_{ks}^j)$ and $\Phi^\prime (O, g; A_i \longrightarrow M_{sk}^j)$  have the properties (\ref{e11}).

Indeed from 1) in Theorem \ref{th6} it follows:
$$A_i A_j \cap M_{ks}^j M_{sk}^i =U_{ij},\ \
M_{ji}^s M_{ij}^k \cap A_k A_s = U_{ks}\ ;$$
From (\ref{e1}) and (\ref{e22}) it follows:
$$A_i M_{ji}^k \cap M_{sk}^j A_s = A_i U_{js} \cap A_j A_s = U_{js}\ ;$$
From (\ref{e2}) and (\ref{e22}) it follows:
$$A_i J \cap M_{ks}^j I= A_i U_{js} \cap A_j A_s = U_{js}\ .$$
Taking into account and (\ref{e3}) we conclude that
$$\Phi(A_j, M_{ji}^s, M_{ij}^k, J) = M_{sk}^i, A_k, A_s, I.$$
From (\ref{e0}), (\ref{e1}) and (\ref{e2}) it follows:
$$
\begin{array}{lll}
\Phi(I^\prime)&=&\Phi(A_iA_s \cap A_jA_k)=\Phi(A_iU_{is} \cap A_jU_{jk})\\
&=&M_{ks}^j U_{is} \cap M_{sk}^i U_{jk}=A_k U_{is} \cap A_s U_{jk}=L^\prime\ .
\end{array}
$$
The proof for $\Phi^\prime$ can be done in a similar fashion.
\end{proof}

Boyan Zlatanov: Plovdiv University ``Paisii Hilendarski'', Faculty of Mathematics and Informatics,
24 "Tsar Assen" str., Plovdiv 4000, Bulgaria

 E-mail address: bzlatanov@gmail.com


\begin{thebibliography}{}

\bibitem{RefC1}
H. S. M. Coxeter, {\it The Real Projective Plane}, McGraw-Hill Book Company Inc., 1949.
\bibitem{RefC2}
H.S. M. Coxeter, {\it Projective Geometry}, Springer-Verlag, New York Inc., 1987.

\bibitem{RefKTZ}
S. Karaibryamov, B. Tsareva, B. Zlatanov, Optimization of the Courses in Geometry by the Usage of Dynamic Geometry Software Sam, {\it The Electronic Journal of Mathematics and Technology}, Volume 7 Number 1, 22-51, (2013).
\bibitem{RefZ}
B. Zlatanov, Some Properties of Reflection of Quadrangle about Point, {\it Annals. Computer Science Series}, 11 (1),79-91, (2013).
\bibitem{RefSh}
I. Sharygin, {\it Problems in Plane Geometry}, Mir Publishers, 1988.
\end{thebibliography}
\end{document}